\documentclass[a4paper,12pt]{article}\usepackage{amssymb}\usepackage{amsmath}\usepackage{verbatim,enumerate}\usepackage[active]{srcltx}\usepackage[dvipdfmx]{graphicx}\usepackage{comment}
\usepackage{amsthm}
\newcommand{\toukou}[1]{\ifx\TOUKOU\undefined\else{#1}\fi}%
\newcommand{\toukoudel}[1]{\ifx\TOUKOU\undefined{#1}\else\fi}%
\newcommand{\toukouchange}[2]{\ifx\TOUKOU\undefined{#1}\else{#2}\fi}%
\usepackage[noadjust]{cite}
 \newtheorem{theorem}{Theorem}[section]
 \newtheorem{proposition}[theorem]{Proposition}
 
 \newtheorem{lemma}[theorem]{Lemma}
 
\theoremstyle{definition}
 
 \newtheorem{remark}[theorem]{Remark}
 
 \newtheorem{example}[theorem]{Example}
\numberwithin{equation}{section}
\numberwithin{figure}{section}

\usepackage[usenames]{color}

\newcommand{\R}{\boldsymbol{R}}
\newcommand{\N}{\boldsymbol{N}}

\newcommand{\coef}{\operatorname{coef}}
\newcommand{\trans}[1]{{\vphantom{#1}}^t{\!#1}}
\renewcommand{\phi}{\varphi}

\newcommand{\inner}[2]{\left\langle{#1},{#2}\right\rangle}

\newcommand{\pmt}[1]{{\begin{pmatrix} #1  \end{pmatrix}}}
\allowdisplaybreaks[4]
\newcommand{\abstractbun}
{Obtaining complete information about the shape of an object
by looking at it from a single direction
is impossible in general.
In this paper, we theoretically study 
obtaining differential geometric information 
of an object from orthogonal projections in
a number of directions.
We discuss relations between 
(1) a space curve and the projected curves 
from several distinct directions, and
(2) a surface and the apparent contours
of projections from several distinct directions,
in terms of differential geometry and
singularity theory.
In particular,
formulae for recovering certain information on the original curves or surfaces
from their projected images
are given.}%
\begin{document}
\begin{center}
{\large {\bf 
Capturing information on curves and surfaces\\
from their projected images}}
\\[2mm]
Masaru Hasegawa,
Yutaro Kabata
and
Kentaro Saji
\begin{abstract}
\abstractbun
\end{abstract}
\end{center}
\renewcommand{\thefootnote}{\fnsymbol{footnote}}

\footnote[0]{
Partly supported by the
JSPS KAKENHI Grants numbered 16J02200 and 18K03301.}

\section{Introduction}\label{sec:intro}
As is well known via triangulation, when we look at a point from two known viewpoints,
we can then calculate where the point is.
Let us turn our attention to the case of a surface.
When we look at a surface,
then we observe an apparent contour 
(a contour),
which gives us some information
about the surface.
In fact, reconstructing objects in 3-space from the information of apparent contours
is an important subject in the area of computer vision, computer graphics
and visual perception
\cite{cg, CB1992, c, multiview, FTP1995, 
GW1987, Koenderink1984, Koenderink1990, kd, Zheng}. 

One cannot obtain complete information
from a finite number of apparent contours, in general.
However,
to obtain the Gaussian curvature of a surface, information 
about the second order derivatives of the surface is required, and in
\cite{Koenderink1984, Koenderink1990}, Koenderink showed 
that one can obtain the Gaussian curvature of a surface 
as the product of the curvature of the contour and the normal curvature 
along a single direction.
While Koenderink's result needs more information than just
the apparent contour, that is, it needs the normal curvature,
this 
fact still suggests 
that we might be able to obtain some information 
about a surface from curvatures of small numbers of contours of the surface. 
It then is natural to ask how much information about the contour
is enough to get the desired information about a surface.

In this paper, we consider an orthogonal projection of $\R^3$ to a plane:
$$
\pi_\xi(x)=x-\inner{x}{\xi}\xi:\R^3\to \xi^\perp
$$
for a unit vector $\xi\in\R^3$.
The map $\pi_\xi$ is called the {\it orthogonal projection 
in the direction\/} $\xi$.
Our interest is in getting local information on surfaces (or curves)
from the curvatures of the contours (or the projected curves) with respect to
orthogonal projections.
In particular, we show how much 
information about contours (or image curves) is enough
to recover the lower degree terms of the Taylor expansions
of surfaces (or curves) at observed points.
Note that we also give explicit formulae for 
reconstructing basic information of curves and surfaces
from a finite number of projected images
(see Remark \ref{apprem}).
In addition, we construct some examples
of sets of different surfaces
whose information with respect to contours
for certain orthogonal projections
is exactly the same 
(Figures \ref{fig:twocontour1}, \ref{fig:twocontour2}, \ref{fig:twocontour3}).
See \cite{cg,c,dgh1,dgh2,dgh3,multiview,kd} 
for other approaches to these kinds
of considerations.

\begin{remark}\label{apprem}
Explicit formulae for reconstructing information about surfaces and curves from 
their projected images are useful tools in practical settings.
In addition to Koenderink's famous result \cite{Koenderink1984, Koenderink1990} 
explained above, \cite{GW1987} provided
a formula for recovering a surface from continuous data of the apparent contours.
Their works have received attention in the context of 
visual perception and computer vision
(cf. \cite{cg, CB1992, c, FTP1995, Zheng}).
The formulae in the present paper are certain kinds of expansions of 
previous results \cite{Koenderink1984, Koenderink1990,GW1987},
which are useful for reconstructing objects from a just few static images. 
We believe that our formulae are ready for use in practical applications
where such a reconstruction is needed.
\end{remark}

Throughout the paper, we use the following notation
for the Taylor coefficients of a given function.
For a $C^\infty$ function $\psi:I\to\R$, we set
$$
(\coef_0(\psi,t,k)=)
\coef(\psi,t,k)=\left(
\psi(0),\psi'(0),\dfrac{\psi''(0)}{2},\ldots,\dfrac{\psi^{(k)}(0)}{k!}
\right)
$$
(${~}'=d/dt$ and $\psi^{(i)}=(\psi^{(i-1)})'$ for $i=1,2,\ldots$),
namely, if $h=a_0+\sum_{i=1}^k(a_i/i!)t^i$,
then $$\coef(\psi,t,k)=(a_0,a_1,a_2/2,\ldots,a_k/k!).$$
The data $\coef(\psi,t,k)$ is called the 
{\it $k$-th order information of\/ $\psi$ $($at\/ $0)$}.
We remark that 
the $k$-th order information of
the given function $\psi$ at $0$
represents 
the $k$-jet of $\psi$ at $0$
in the terminology of singularity theory
(cf. \cite{Izumiya-book}).

\subsection{Projections of curves}
Let $I$ be an open interval containing $0$,
and let $\gamma: I \to \R^3$ ($\gamma(0)=(0,0,0)$)
be a given unknown regular $C^\infty$ curve 
whose curvature does not vanish at $0$.
We remark that $\gamma$ has the orientation induced from that of $I$.
Rotating the coordinate system of $\R^3$ if necessary, 
for any $k\in\N$,
we may assume that $\gamma$ is locally written around $0$ as
\begin{equation}\label{eq:gamma}
\gamma(t)
=\left(
t,\ \sum_{i=2}^k \dfrac{a_i}{i!}t^i,\ 
\sum_{i=3}^k \dfrac{b_i}{i!}t^i\right)+(O(k+1),O(k+1),O(k+1)),
\end{equation}
where $a_i,b_i\in \R$ $(i=2,\ldots,k)$,
and $O(k+1)$ stands for the terms whose degrees are
greater than $k$.
Specifically, $a_2$ and $b_3$ are important values of the space curve:
the curvature and the torsion at $0$.
Set $\gamma_\xi=\pi_\xi\circ \gamma$
for a unit vector $\xi\in \R^3$.
Our aim is to investigate how many conditions are enough
to recover the above coefficients
in terms of the curvatures of $\gamma_{\xi}$ using
a number of distinct directions $\xi$.

Since the setting is complicated for general
choices of projection directions,
we focus on the two singular cases
where the kernel direction $\xi$ of an orthogonal projection
is geometrically restricted.
The following are our settings, and also abstracts of the results
which will be given in Section \ref{sec:cur}:

\begin{itemize}
\item
We take two linearly independent vectors
$\xi_1,\xi_2$, where each
projected curve $\gamma_{\xi_i}$ $(i=1,2)$ has 
an inflection point at $0$.
This implies that $\xi_1,\xi_2$ lie in the osculating plane,
with the exception of the tangent line of $\gamma$
(Figure \ref{fig:oricurve1}).
Then the coefficients $a_i, b_{i}$
can be uniquely determined from the certain order of the information of 
the curvature functions of 
$\gamma_{\xi_i}$ $(i=1,2)$ at $0$ (Theorem \ref{prop:211}).
Namely, the coefficients $a_i, b_i$
can be uniquely determined by the information of 
the derivatives of the curvatures of the two projected curves.

\item
We take $\xi_1$ as a tangent vector of $\gamma$ at $0$.
Then,
$\gamma_{\xi_1}=\pi_{\xi_1}\circ\gamma$ has a singular point at $0$
(Figure \ref{fig:oricurve2}).
We also take another vector $\xi_2$.
Then the coefficients $a_i, b_{i}$ ($i\le 5$)
can be uniquely determined by the information of 
the curvature functions of 
$\gamma_{\xi_i}$ $(i=1,2)$ at $0$.
Namely, the coefficients $a_i, b_i$ ($i\le 5$)
can be uniquely determined from the curvature functions of two projected curves
from the tangential direction and another direction (Theorem \ref{prop:212}).
The notion of the cuspidal curvature of a singular 
plane curve (introduced in \cite{suyo}), especially, 
plays an important role.
\end{itemize}

\subsection{Projections of surfaces}
Let $U$ be an open subset of $\R^2$ containing $0=(0,0)$,
and let $f: U \to \R^3$ ($f(0)=(0,0,0)$)
be a given unknown regular $C^\infty$ surface.
Without loss of generality,
we may assume that
$f$ is given by
\begin{equation}\label{eq:surfsiki}
f(u,v)=\big(u,v,h(u,v)\big),\quad
h(u,v)=
 \dfrac{a_{20}}{2}u^2
+\dfrac{a_{02}}{2}v^2
+\sum_{i+j=3}^k \dfrac{a_{ij}}{i!j!}u^iv^j+O(k+1),
\end{equation}
where $a_{ij}\in \R$ $(i,j=0,1,2,\ldots,k)$.
We call
$a_{20},a_{02}$ (respectively, $a_{30}$, $a_{21}$, $a_{12}$, 
$a_{03}$)
the {\it second order\/} (respectively, the {\it third order\/}) 
information of $f$ at $0$.
Taking vectors which are tangent to the
image of $f$ at $0$, we consider
apparent contours of $f$ projected from
the directions of these vectors.
For tangent vector $\xi$ of the image of $f$ at $0$,
we set $f_\xi=\pi_\xi\circ f$.
We call the set $S$ of singular points 
the {\it contour generator\/} ({\it with respect to\/ $\xi$}), 
and $f_\xi(S)$ 
the {\it contour\/} ({\it of\/ $f$ with respect to\/ $\xi$}).

The following are abstracts of the results on surfaces
which will be given in Section \ref{sec:surf}:
\begin{itemize}
\item 
We take three
``general'' (respectively, four ``general'') distinct directions.
Then the second order (respectively, the third order) 
information of a surface 
is 
uniquely determined by the $0$-th order (respectively, 
the first order)
information of the curvatures of
the contours with respect to the directions.
Moreover, 
formulae on the relations between the 
information of the surfaces and the curves
are explicitly given (Theorem \ref{thm:threegandm} (respectively, 
the formula \eqref{eq:a3121})).
See \eqref{eq:sinsinsin} for the meaning of general
distinct directions.
We remark that knowing the second order information
is the same as knowing the pair of
values of the mean and Gaussian curvatures.

\item
We give an example of a pair of different surfaces
having the same information from the
curvatures of contours with respect to
two distinct directions.
Namely,
a surface $f$ with
two distinct directions $(\xi_1,\xi_2)$
and another surface $\tilde f$ with
two distinct directions 
$(\tilde \xi_1,\tilde \xi_2)$ are constructed,
such that
the information on the two contours of $f$ with respect to
$\xi_1$ and $\xi_2$ is
the same as 
the information on the two contours of $\tilde f$ with respect to
$\tilde\xi_1$ and $\tilde\xi_2$
(Example \ref{ex:contoursame} and 
Figures \ref{fig:twocontour1}, \ref{fig:twocontour2}, 
\ref{fig:twocontour3}).

\item We show that if the Gaussian curvature
is positive, then there exist two directions
such that the product of the contours of
these directions gives the Gaussian curvature
(Proposition \ref{prop:contconj}).
\end{itemize}

\begin{remark}
From the above results,
we see that
in order to judge 
the sign of the Gaussian at an observed point of a given surface,
looking at it along general three directions
in the tangent plane
is necessary and sufficient.
\end{remark}

\section{Projections of space curves}\label{sec:cur}
Let $\gamma: I \to \R^3$ 
be a
$C^\infty$ curve ($\gamma(0)=(0,0,0)$), and let
$\gamma_\xi=\pi_\xi\circ \gamma$
for $\xi$ with $\pi$ given as in the introduction.
We assume that the curvature of $\gamma$ 
does not vanish at $0$.
We consider the following two cases.
The first case is that the projection curve $\gamma_\xi$ has 
an inflection point,
namely, the vector $\xi$ lies in the osculating plane.
The second case is that one of the 
projection curve $\gamma_\xi$ has 
a singular point,
namely, the vector $\xi$ is tangent to $\gamma$ at $0$.

\subsection{Projections in the osculating plane}\label{oscu}
In this section, we consider the case that 
the curvature of $\gamma$ does not vanish and
$\gamma_\xi$ has an inflection point at $0$.
Then it holds that
$\xi$ lies
in the osculating plane, except for the tangent line of $\gamma$ at $0$.
We remark that $\gamma$ has the orientation induced from that of $I$.
Then rotating the coordinate system of $\R^3$ if necessary, 
we may assume that $\gamma$ is written as in
\eqref{eq:gamma}
and
$\xi(\theta_j)=(\cos\theta_j,\sin\theta_j,0)$,
where $0<\theta_j<\pi$ $(j=1,2,\ldots)$.
We give the orientation of $\xi^\perp$ as follows:
We take a basis $\{X,Y\}$ of $\xi^\perp$.
We say that $\{X,Y\}$ is a positive basis if
$\{X,Y,\xi\}$ is a positive basis of $\R^3$.
We set the orientation of $\pi_{\xi(\theta_j)}\circ \gamma$
to agree with that of $\gamma$ (see Figure \ref{fig:oricurve1}).
We set $\pi_{\xi(\theta_j)}\circ\gamma=\gamma_{\theta_j}$,
and also we set
$s_j$ to be the arc-length of $\gamma_{\theta_j}$,
and set $\kappa_{\theta_j}$ to be the curvature of 
$\gamma_{\theta_j}\subset \xi^\perp$
as a curve in the oriented plane $\xi^\perp$.

\begin{figure}[!htb]
\centering
\includegraphics[scale=0.4, clip]{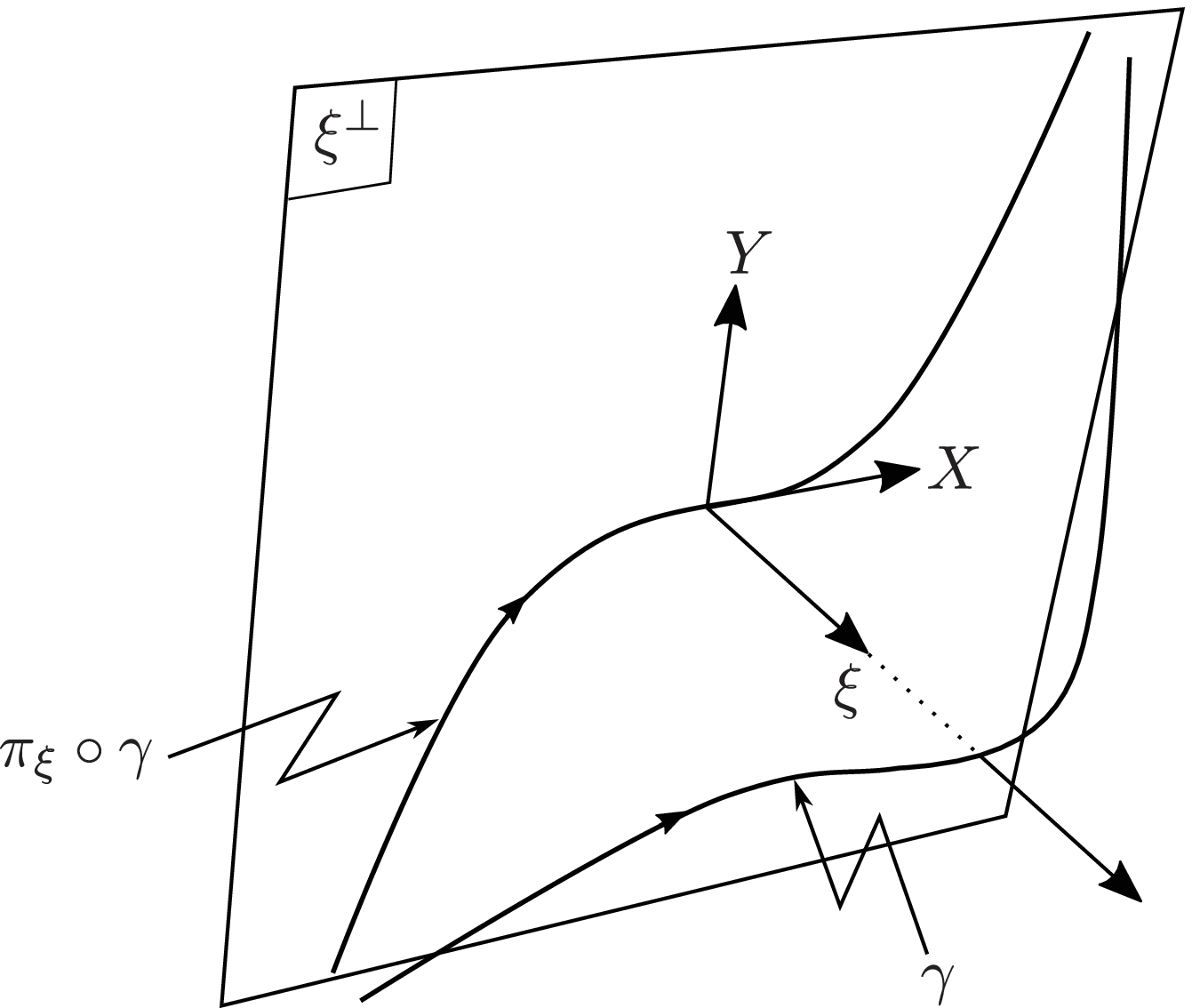}
\caption{Orientations of $\xi^{\perp}$ and $\pi_{\xi}\circ \gamma$.} 
\label{fig:oricurve1}
\end{figure}

Suppose that we are given the information of the curvatures of the contours from
two distinct directions $0<\theta_1,\theta_2<\pi$ 
($\theta_1\ne\theta_2$). 
Set $\phi=\theta_1-\theta_2$
and $\tilde\kappa_{\theta_j}=(d\kappa_{\theta_j}/ds_j(0))^{1/3}$.
We can determine the local information of space curves from
local information of projected curves
in two distinct directions:
\begin{theorem}\label{prop:211}
If\/
$(d\kappa_{\theta_1}/ds_1(0),d\kappa_{\theta_2}/ds_2(0))
\ne(0,0)$,
then the following hold.
\begin{enumerate}
\item $\theta_1, \theta_2$ and\/ $b_3$ are uniquely determined
from the second order information of\/ 
$\kappa_{\theta_1}, \kappa_{\theta_2}$ at\/ $0$
and\/ $\phi=\theta_1-\theta_2$. 
\item 
The coefficients\/
$a_{n-2}$ and\/ $b_n$ are uniquely determined
from the\/ $(n-1)$-st order information of\/ 
$\kappa_{\theta_1}, \kappa_{\theta_2}$ at\/ $0$
and\/ $\phi$
for\/ $n\ge4$. 
\end{enumerate}
\end{theorem}
The explicit formulae for $\theta_1,\theta_2=\theta_1+\phi$
and $b_3$
are given as in \eqref{eq:theta1ex} and 
\eqref{eq:b3explicit} in the proof.
\begin{proof}
Let $e_1,e_2$ be an orthonormal basis of the osculating plane
of $\gamma$ at $0$.
Then the curvature $\kappa_{\theta_i}$ for
a general parameter $t$ is
$$
\kappa_{\theta_i}(t) = 
\dfrac{
\det\left(
\pmt{\gamma'(t)\cdot e_1\\ \gamma'(t)\cdot e_2},
\pmt{\gamma''(t)\cdot e_1\\ \gamma''(t)\cdot e_2}\right)}
{((\gamma'(t)\cdot e_1)^2+(\gamma'(t)\cdot e_2)^2)^{3/2}}
 = 
\dfrac{\det\big(\gamma'(t), \gamma''(t), \xi_{\theta_i}\big)}
{|\gamma_{\theta_i}'(t)|^3}\ 
\left('=\dfrac{d}{dt}\right).$$
We set $\alpha(t)=\det\big(\gamma'(t), \gamma''(t), \xi_{\theta_i}\big)$
and $\beta(t)=\gamma_{\theta_i}'(t)\cdot \gamma_{\theta_i}'(t)$.
Then noticing $\alpha(0)=0$, we have
\begin{align}
\kappa_{\theta_i}'(0) =& \alpha'(0)/\beta(0)^{3/2}\nonumber\\
\kappa_{\theta_i}''(0) =&
(-3 \alpha'(0) \beta'(0) + 
 \beta(0) \alpha''(0))/\beta(0)^{5/2}\label{eq:kappaibibun}\\
\kappa_{\theta_i}'''(0) =&
\dfrac{1}{4\beta(0)^{7/2}}
\Big(9 \alpha'(0) (5 \beta'(0)^2 - 2 \beta(0) \beta''(0)) \nonumber\\
&-18\alpha''(0)\beta(0)\beta'(0)+4\alpha'''(0)\beta(0)^2\Big)
\nonumber
\end{align}
and
\begin{align}
\alpha'(0) =& \det(\gamma'(0),\gamma'''(0),\xi_{\theta_i})\nonumber\\
\alpha''(0) =& \det(\gamma''(0),\gamma'''(0),\xi_{\theta_i})
+\det(\gamma'(0),\gamma^{(4)}(0),\xi_{\theta_i})\label{eq:abibun}\\
\alpha'''(0) =& 2\det(\gamma''(0),\gamma^{(4)}(0),\xi_{\theta_i})
+\det(\gamma'(0),\gamma^{(5)}(0),\xi_{\theta_i}).\nonumber
\end{align}
Since $0<\theta_{i}<\pi$, it holds that 
$\sin\theta_{i}\ne0$,
and by a direct calculation we have
\begin{align}\label{eq:coefkappa00}
&\coef(\kappa_{\theta_i},t,3)\\
=&
\Bigg(
0,-\dfrac{b_3}{\sin^2\theta_i},
-\dfrac{b_4\sin\theta_i+5 a_2 b_3 \cos\theta_i}{2\sin^3\theta_i},\nonumber\\
&-\dfrac{1}{6\sin^4\theta_i}
\Big(b_5\sin^2\theta_i\nonumber\\
&+(9 a_3 b_3+7 a_2 b_4) \cos\theta_i\sin\theta_i
+27 a_2^2 b_3\cos^2\theta_i\Big)
\bigg)\nonumber
\end{align}
for $i=1,2$.
Since $s_i$ is an arc-length parameter of $\gamma_{\theta_i}$,
$$
s_i(t)=\int^t_0 \beta(t)^{1/2}\,dt.
$$
We set $t(s_i)$ to be the inverse function of the above $s_i(t)$.
Then since $s_i'(t)=\beta(t)^{1/2}$, we see that
\begin{align*}
\dfrac{dt}{ds_i}(s_i(t))
=&
\dfrac{1}{s_i'(t)}
=
\dfrac{1}{\beta(t)^{1/2}},\\
\dfrac{d^2t}{ds_i^2}(s_i(t))
=&
-\dfrac{s_i''(t)}{(s_i'(t))^3}
=
-\dfrac{\beta'(t)}{2\beta(t)^2},\\
\dfrac{d^3t}{ds_i^3}(s_i(t))
=&
\dfrac{3(s_i''(t))^2-s_i'(t)s_i'''(t)}{(s_i'(t))^5}
=
\dfrac{2(\beta'(t))^2-\beta(t)\beta''(t)}{2\beta^{7/2}(t)}.
\end{align*}
On the other hand, by the definition of $\beta$,
$$
\coef(\beta,t,2)
=
\big(\sin^2\theta_i,
-2a_2 \sin\theta_i\cos\theta_i,
\cos\theta_i (a_2^2 \cos\theta_i-a_3 \sin\theta_i)
\big)
$$
holds.
Thus
\begin{align}
&\coef\left(
t(s_i),s_i,3\right)
=
\bigg(0,\dfrac{1}{\sin\theta_i},
\dfrac{a_2\cos\theta_i}{2\sin^3\theta_i},\label{eq:tsicoef}\\
&\hspace{5mm}
\dfrac{\cos\theta_i}{6\sin^5\theta_i}(3 a_2^2 \cos\theta_i+a_3 \sin\theta_i)
\bigg).
\nonumber
\end{align}
Moreover, by
\begin{align*}
\dfrac{d\kappa_{\theta_i}}{ds_i}(s_i)&=
\kappa_1'(t(s_i)) \dfrac{dt}{ds_i}(s_i)\\
\dfrac{d^2\kappa_{\theta_i}}{ds_i^2}(s_i)&=
 \kappa_1''(t(s_i)) \left(\dfrac{dt}{ds_i}(s_i)\right)^2+ 
 \kappa_1'(t(s_i)) \dfrac{dt^2}{d^2s_i}(s_i)\\
\dfrac{d^3\kappa_{\theta_i}}{ds_i^3}(s_i)&=
\kappa_1'''(t(s_i))\left(\dfrac{dt}{ds_i}(s_i)\right)^3 
+3  \kappa_1''(t(s_i)) \dfrac{dt}{ds_i}(s_i)\dfrac{dt^2}{d^2s_i}(s_i) \\
&\hspace{10mm}
+ \kappa_1'(t(s_i)) 
\dfrac{dt^3}{d^3s_i}(s_i)
\end{align*}
together with \eqref{eq:kappaibibun}, \eqref{eq:abibun},
\eqref{eq:coefkappa00} and \eqref{eq:tsicoef},
we have
\begin{align}
\label{eq:coefkappa01}
&\coef(\kappa_{\theta_{i}},s_{i},3)\\
=&
\Bigg(0,
-\dfrac{b_{3}}{\sin^3\theta_{i}},
-\dfrac{b_{4}\sin\theta_{i}+6 a_{2} b_{3} \cos\theta_{i}}
{2\sin^5\theta_{i}},\nonumber\\
&
-\dfrac{1}{6\sin^7\theta_{i}}
\Big(45 a_{2}^2 b_{3} \cos^2\theta_{i}+b_{5} \sin^2\theta_{i}\nonumber\\
&\hspace{5mm}+10 (a_{3} b_{3}+a_{2} b_{4}) \sin\theta_{i}\cos\theta_{i}
\Big)\Bigg).\nonumber
\end{align}
By \eqref{eq:coefkappa01}, the condition
$d\kappa_{\theta_1}/ds_1(0)\ne0$ or $d\kappa_{\theta_2}/ds_2(0)
\ne0$
implies $b_3\ne0$.
Thus 
$(d\kappa_{\theta_1}/ds_1(0),d\kappa_{\theta_2}/ds_2(0))
\ne(0,0)$
implies
$d\kappa_{\theta_1}/ds_1(0)\ne0$ and
$d\kappa_{\theta_2}/ds_2(0)\ne0$.
Taking another direction $\theta_2=\theta_1+\phi$ 
$(0<\theta_2<\pi)$,
we may consider $\kappa_{\theta_1},\kappa_{\theta_2},\phi$
to be known. 
Since the equation
$$
\dfrac{d\kappa_{\theta_1}/ds_1(0)}{d\kappa_{\theta_2}/ds_2(0)}
=
\dfrac{\sin^3(\theta_1+\phi)}{\sin^3\theta_1}
$$
can be solved as
\begin{equation}\label{eq:theta1ex}
\theta_1=\cot^{-1}
\left(
\dfrac{
\left(
\dfrac{d\kappa_{\theta_1}/ds_1(0)}
{d\kappa_{\theta_2}/ds_2(0)}\right)^{1/3}
-\cos\phi}
{\sin\phi}\right)\in(0,\pi),
\end{equation}
we obtain $\theta_1$ and $\theta_2$.

Furthermore, by \eqref{eq:coefkappa01}, it holds that
\begin{equation}\label{eq:coefkappa02}
\sin \theta_i =-\frac{\tilde{b}}{\tilde{\kappa}_{\theta_i}},\quad(i=1,2)
\end{equation}
where
$\tilde b=b_3^{1/3}$ and
$\tilde\kappa_{\theta_i}=(d\kappa_{\theta_i}/ds_i(0))^{1/3}$.
Substituting \eqref{eq:coefkappa02} into the trigonometric identity
$$
\cos^2{(\theta_1-\theta_2)}+\sin^2{\theta_1}+\sin^2{\theta_2}-2\sin{\theta_1}\sin{\theta_2}\cos{(\theta_1-\theta_2)}-1
=0,
$$
we have
\begin{equation}\label{eq:coefkappa03}
(\tilde\kappa_{\theta_1}^2
-2\cos{\varphi}\;
\tilde\kappa_{\theta_1}\tilde\kappa_{\theta_2}
+\tilde\kappa_{\theta_2}^2)\tilde b^2
-
\sin^2{\varphi}\; \tilde\kappa_{\theta_1}^2\tilde\kappa_{\theta_2}^2=0.
\end{equation}
Since $\theta_1\ne\theta_2$, it holds that $\sin\phi\ne0$.
Thus by $\tilde\kappa_{\theta_1}\tilde\kappa_{\theta_2}\ne0$,
we see that the equality \eqref{eq:coefkappa03} 
implies 
$\tilde\kappa_{\theta_1}^2
-2\cos{\varphi}\;
\tilde\kappa_{\theta_1}\tilde\kappa_{\theta_2}
+\tilde\kappa_{\theta_2}^2\ne0$.
Thus
\eqref{eq:coefkappa03} also implies that
we obtain $b_3$, 
since $\tilde\kappa_{\theta_1}$, 
$\tilde\kappa_{\theta_2}$, $\phi$ are known. In fact,
\begin{equation}\label{eq:b3explicit}
b_3=
\left(
\dfrac{
\sin^2{\varphi}\; \tilde\kappa_{\theta_1}^2\tilde\kappa_{\theta_2}^2
}{
\tilde\kappa_{\theta_1}^2
-2\cos{\varphi}\;
\tilde\kappa_{\theta_1}\tilde\kappa_{\theta_2}
+\tilde\kappa_{\theta_2}^2
}\right)^{3/2},\quad
\left(
\tilde\kappa_{\theta_i}=
\left(\dfrac{d\kappa_{\theta_i}}{ds_i}(0)\right)^{1/3}\quad(i=1,2)
\right).
\end{equation}
Thus the claim (1) holds.

Since the pair of equations
$$
\dfrac{d^2\kappa_{\theta_i}}{ds_i^2}(0)=
-\dfrac{b_{4}\sin\theta_i+6 a_{2} b_{3} \cos\theta_i}
{2\sin^5\theta_i}
\quad(i=1,2)
$$
is a linear system for $a_2,b_4$, by $\theta_1\ne\theta_2$
we obtain $a_2$ and $b_4$ from \eqref{eq:coefkappa01} if $b_3\ne0$.
Thus the claim (2) when $n=4$ holds.
Next, we show (2) when $n\geq5$.
We assume that
$\phi$,
$a_1,\ldots,a_{n-3}$ and $b_1,\ldots,b_{n-1}$ are known
for $n\geq5$.
Now we take the $(n-2)$-th derivative of $\kappa_{\theta_i}$ 
with respect to $s_i$.
That value of the derivative at $0$ is written in terms of 
$a_1,\ldots,a_{n}$ and $b_1,\ldots,b_{n}$.
In the formula of $d^{n-2}d\kappa_{\theta_i}/ds_i^{n-2}(0)$,
as a polynomial of 
$a_1,\ldots,a_{n},b_1,\ldots,b_{n}$,
we show that $a_{n-1},a_n$ do not appear,
and $a_{n-2},b_n$ appear only to the first power linear terms.
We set $\kappa_{\theta_i}=\kappa$ and $s_i=s$ for
simplicity, for the moment.
Let $t(s)$ be the inverse function of
$$
s=\int_0^t |\gamma_{\theta_i}'(t)|\,dt,
$$
where $'=d/dt$.
By the formula for differentiation of a composition
of functions (see \cite[($3_n$)]{mckie} or \cite[(3.56)]{gould} for example),
\begin{equation}\label{eq:comp}
\dfrac{d^{n-2}}{ds^{n-2}}
\kappa(t(s))
=
\sum_{k=0}^{n-2}
\dfrac{d^k\kappa}{dt^k}(t(s))
\dfrac{(-1)^k}{k!}
\sum_{j=0}^k
(-1)^j\pmt{k\\j} t(s)^{k-j}
\dfrac{d^{n-2}}{ds^{n-2}}(t(s)^j),
\end{equation}
and we look at the terms
$$
\dfrac{d^jt(s)}{ds^j}(0),\quad
\dfrac{d^j\kappa}{dt^j}(0),\quad
(j=n-4,n-3,n-2).
$$
Since
$dt/ds=|\gamma_{\theta_i}'|^{-1}$,
it holds that
\begin{align*}
\dfrac{d^{j}t}{ds^{j}}
&=
\dfrac{d^{j-1}}{ds^{j-1}}\left(
\dfrac{1}{|\gamma_{\theta_i}'(t(s))|}\right)
=
\dfrac{d^{j-1}}{ds^{j-1}}\left(
\inner{\gamma_{\theta_i}'(t(s))}
{\gamma_{\theta_i}'(t(s))}^{-1/2}
\right)\\
&=
-\dfrac{d^{j-2}}{ds^{j-2}}\left(
\inner{\gamma_{\theta_i}'(t(s))}
{\gamma_{\theta_i}'(t(s))}^{-3/2}
\inner{\gamma_{\theta_i}'(t(s))}
{\gamma_{\theta_i}''(t(s))}
\dfrac{dt(s)}{ds}
\right)\\
&=
-\dfrac{d^{j-2}}{ds^{j-2}}\left(
\inner{\gamma_{\theta_i}'(t(s))}
{\gamma_{\theta_i}'(t(s))}^{-2}
\inner{\gamma_{\theta_i}'(t(s))}
{\gamma_{\theta_i}''(t(s))}
\right).
\end{align*}
Continuing to calculate the derivative of the function by $s$,
we finally obtain a polynomial consisting of terms with 
some powers of
$\inner{\gamma_{\theta_i}'(t(s))}{\gamma_{\theta_i}'(t(s))}$
and
$$
\inner{\gamma_{\theta_i}^{(\ell_1)}(t(s))}
{\gamma_{\theta_i}^{(\ell_2)}(t(s))}\quad
(1\le\ell_1\le \ell_2\le j,\ 3 \le \ell_1+\ell_2 \le j+1).$$
This implies that
$a_{n-1},a_n,b_n$ do not appear in
$(d^jt(s)/ds^j)(0)$ $(j=n-4,n-3,n-2)$.
Moreover, although
$a_{n-2}$ may appear in
$(d^{n-2}t(s)/ds^{n-2})(0)$
from the term
$
\langle\gamma_{\theta_i}',\gamma_{\theta_i}^{(n-2)}\rangle,
$
it does not actually appear, since
$\gamma_{\theta_i}'(0)=(1,0,0)$
and the first component of $\gamma_{\theta_i}^{(n-2)}(0)$ is $0$.
On the other hand,
since
$\gamma$ is given by \eqref{eq:gamma},
\begin{equation}\label{eq:gammadiff}
\begin{split}
\det(\gamma',\gamma^{(n-2)},\xi_{\theta_i})
=&- b_{n-2}\sin\theta_i,\ 
\det(\gamma',\gamma^{(n-1)},\xi_{\theta_i})
=- b_{n-1}\sin\theta_i,\\
\det(\gamma',\gamma^{(n)},\xi_{\theta_i})
=&- \boxed{b_{n}}\sin\theta_i,\\
\det(\gamma'',\gamma^{(n-2)},\xi_{\theta_i})
=&a_2b_{n-2}\cos\theta_i,\ 
\det(\gamma'',\gamma^{(n-1)},\xi_{\theta_i})
=a_2b_{n-1}\cos\theta_i,\\
\det(\gamma''',\gamma^{(n-2)},\xi_{\theta_i})
=&(a_3b_{n-2}-b_3\boxed{a_{n-2}})\cos\theta_i,
\end{split}
\end{equation}
and
$a_{n-2},a_{n-1},a_n,b_n$ do not appear in 
$\det(\gamma',\gamma'',\xi_{\theta_i})^{(j)}(0)$
$(j=1,\ldots,n-3)$.
This implies that they also do not appear in
$d^{j}\kappa/dt^{j}$ $(j=1,\ldots,n-3)$.
Furthermore, since
\begin{equation}\label{eq:diff3}
\begin{split}
\dfrac{d^{n-2}\kappa}{dt^{n-2}}
=&
\sum_{j+k\leq n}
c_{jk}
\left(\dfrac{1}{|\gamma'_{\theta_i}|^3}\right)^{(n-2-j-k)}
 \det(\gamma^{(j)},\gamma^{(k)},\xi_{\theta_i})\\
&\hspace{10mm}+
\dfrac{1}{|\gamma'_{\theta_i}|^3}\bigg(
\det(\gamma',\gamma^{(n)},\xi_{\theta_i})
+
(n-3)
\det(\gamma'',\gamma^{(n-1)},\xi_{\theta_i})\\
&\hspace{30mm}+
m\det(\gamma''',\gamma^{(n-2)},\xi_{\theta_i})
\bigg)
\end{split}
\end{equation}
$(n\geq5)$ holds, where 
$m=(n-5)(n-2)/2$ and
$c_{jk}$ are natural numbers,
this implies that
$a_{n-1},a_n$ do not appear in \eqref{eq:comp}.
Moreover, 
since the coefficient of 
$d^{n-2}\kappa/dt^{n-2}$ in \eqref{eq:comp}
is
$(dt/ds)^{n-2}$,
setting 
$$
A_i=\dfrac{1}{|\gamma'_{\theta_i}(0)|^{3}}
\left(\dfrac{dt}{ds}(0)\right)^{n-2},
$$
the equations \eqref{eq:comp} at $s=0$ for $i=1,2$ are
\begin{equation}\label{eq:projeq}
\begin{split}
\dfrac{d^{n-2}\kappa_{\theta_1}}{ds_1^{n-2}}(0)
=
-A_1(b_n\sin\theta_1+mb_3a_{n-2}\cos\theta_1)+B_1,\\
\dfrac{d^{n-2}\kappa_{\theta_2}}{ds_2^{n-2}}(0)
=
-A_2(b_n\sin\theta_2+mb_3a_{n-2}\cos\theta_2)+B_2,
\end{split}
\end{equation}
where $B_i$ $(i=1,2)$ are terms consisting of
$a_1,\ldots,a_{n-3}$ and $b_1,\ldots,b_{n-1}$.
The equation \eqref{eq:projeq} can be solved when
$A_1A_2b_3\sin(\theta_1-\theta_2)\ne0$.
Since 
$dt/ds(0)=|\gamma_{\theta_i}'(0)|^{-1}$,
we have the assertion.
\end{proof}
Since obtaining $a_2,a_3,b_3$ is equivalent to
obtaining the curvature, the first derivative of
the curvature and the torsion,
results of this type for the perspective
projection can be found in \cite{c,lz} and \cite[Theorem 8]{multiview}.
Since we can detect the coefficients of the Taylor expansion 
of $\gamma$, using our result,
one can easily construct the desired curve
whose projections have the prescribed 
curvatures.

\subsection{Projections in the tangential direction and another direction}\label{tan}
In this section, we consider the case that $\xi_1$ is tangent to
$\gamma$ at $0$.
In this case, $\gamma_{\xi_1}=\pi_{\xi_1}\circ\gamma$ has a singular point at $0$.
To consider differential geometric invariants of the singular curve,
we describe the cuspidal curvature of singular plane curves introduced
in \cite{suyo} (see also \cite{ss}).
Let $c:I\to\R^2$ be a plane curve, and $c'(0)=0$.
The curve $c$ is said to be {\it $A$-type\/} at $0$
if $c''(0)\ne0$.
Let $c$ be an $A$-type curve at $0$. 
Then
$$
\mu
=
\dfrac{\det(c''(0),c'''(0))}{|c''(0)|^{5/2}}.
$$
does not depend on the choice of parameter,
and is called the {\it cuspidal curvature}.

Let $\gamma:I\to\R^3$ be a
$C^\infty$ curve with non-zero curvature at $0$.
We assume that 
$\gamma_{\xi_1}$ has a singular point at $0$.
Since the curvature of $\gamma$ does not vanish,
by the Frenet formula,
$\gamma_{\xi_1}(0)$
is an $A$-type curve at $0$.
We also assume that there exists an integer $N$ such that
$\det((\gamma_{\xi_1})'',
(\gamma_{\xi_1})^{(i)})(0)=0$ ($i=3, 5,\ldots,2N-1$)
and
$\det((\gamma_{\xi_1})'',
(\gamma_{\xi_1})^{(2N+1)})(0)\ne0$.
We give the positively oriented $xyz$-coordinate 
system for $\R^3$, and rotating this coordinate system,
we give a $yz$-coordinate system for 
$\xi_1^\perp$ as follows.
We set the $y$-axis 
as the direction
of $(\gamma_{\xi_1})''(0)$,
and set the $x$-axis as the direction of $\xi_1$.
We give an orientation of $\gamma_{\xi_1}$ so that
$\det((\gamma_{\xi_1})'',
(\gamma_{\xi_1})^{(2N+1)})(0)>0$,
and also that of $\gamma$ agrees with that of 
$\gamma_{\xi_1}$ (see Figure \ref{fig:oricurve2}).
Then we may assume that $\gamma$ is given by \eqref{eq:gamma}
with $a_2>0,b_3\geq0$.
Then $\mu=b_3/a_2^{3/2}$.
On the other hand,
we consider a unit vector 
$\xi_2=(\sin\theta_1 \cos\theta_2, \sin\theta_1 \sin\theta_2, \cos\theta_1)$ 
$(\cos\theta_1\ne0)$
which is not tangent to $\gamma$ at $0$. 
Since we take the above $xyz$-coordinate, $\theta_1,\theta_2$ are
known.

\begin{figure}[!htb]
\centering
\includegraphics[scale=1, clip]{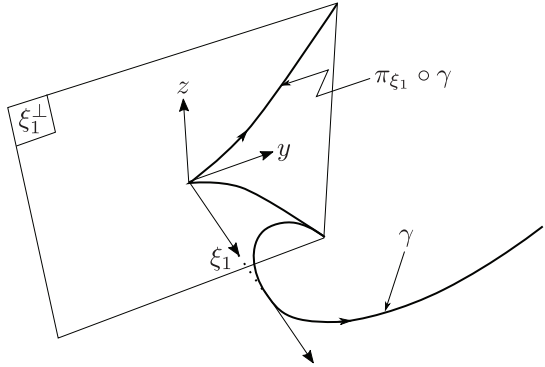}
\caption{Orientations of $\xi_1^\perp$ and $\pi_{\xi_1}\circ \gamma$.}
\label{fig:oricurve2}
\end{figure}

\begin{theorem}\label{prop:212}
Suppose
that\/
$\cos\theta_1\ne0$, 
then the following hold.
\begin{enumerate}
\item 
The coefficients\/ $a_2$ and\/ $b_3$ are uniquely determined
by the cuspidal curvature\/ $\mu$ and 
the\/ $0$-th order information of\/ $\kappa_{\xi_2}$ at\/ $0$.
\item In addition to\/ {\rm (1)},
$a_3$ is uniquely determined
by the cuspidal curvature\/ $\mu$ and 
the first order information of\/ $\kappa_{\xi_2}$ at\/ $0$.
\end{enumerate}
\end{theorem}
\begin{proof}
We remark that 
$\cos\theta_1\ne0$ implies
$\cos^2\theta_1 \cos^2\theta_2+\sin^2\theta_2\ne0$.
The curvature $\kappa_{\xi_2}$ of the plane curve
$\gamma_{\xi_2}$ satisfies
\begin{align}\label{eq:tangene}
\coef(\kappa_{\xi_2},s,1)
=&
\Bigg(
\dfrac{a_{2} \cos\theta_1}
{(\cos^2\theta_1 \cos^2\theta_2+\sin^2\theta_2)^{3/2}}
,\;\;
\dfrac{Q(\theta)}
{(\cos^2\theta_1 \cos^2\theta_2+\sin^2\theta_2)^3}\Bigg),
\end{align}
where
\begin{align*}
&Q(\theta)=-b_{3} \cos^2\theta_1 \cos^2\theta_2 \sin\theta_1 \sin\theta_2
-b_{3} \sin\theta_1 \sin^3\theta_2\\
&\hspace{5mm}
+\cos^3\theta_1 \cos\theta_2 
(a_{3} \cos\theta_2-3 a_{2}^2 \sin\theta_2)
+\cos\theta_1 \sin\theta_2 (3 a_{2}^2 \cos\theta_2+a_{3} \sin\theta_2).
\end{align*}
Since we know $\mu=b_3/a_2^{3/2}$ and $\theta_1,\theta_2$,
we obtain $a_2$ and $b_3$ from the first component of \eqref{eq:tangene}
if $\cos\theta_1\ne0$.
Furthermore, we also obtain 
$a_3$ from the second component of \eqref{eq:tangene}, under
the assumption 
$\cos\theta_1(\cos^2\theta_1 \cos^2\theta_2+\sin^2\theta_2)\ne0$.
\end{proof}

See \cite[Section 3]{koganolver} for relationships 
between invariants of space curves and
projected plane curves.

\section{Projections of surfaces}\label{sec:surf}
In this section, we consider
the local information of surfaces
and contours.
Let $U\subset \R^2$ be a neighborhood of the origin $0=(0,0)$.
Let $f:U\to\R^3$ ($f(0)=(0,0,0)$) be a $C^\infty$ surface
with non-vanishing Gaussian curvature at $0$.
We assume that $0$ is not an umbilical point.
Without loss of generality,
we may assume that $f$ is written in the form
\eqref{eq:surfsiki}
with $a_{20}a_{02}\ne0$, $a_{20}>a_{02}$, $a_{20}>0$.
We set the unit normal vector $\nu$ of $f$ so that it satisfies
$\nu(0,0)=(0,0,1)$.

\subsection{Obtaining information about surfaces from contours}\label{sec:surfnp}
Let $\xi$ be a unit vector which is tangent to $f$ at $0$.
Then we may assume $\xi=\xi(\theta_1)=(\cos\theta_1,\sin\theta_1,0)$,
where $0<\theta_1<\pi$.
The set $S$ of singular points of the map 
$\pi_{\xi(\theta_1)}\circ f$
is 
\begin{equation}\label{eq:projsingcond}
S=\{(u,v)\,|\,\cos\theta_1h_u+\sin\theta_1h_v=0\},
\end{equation}
where $h$ is given in \eqref{eq:surfsiki}.
We assume
$$
p(\theta_1)
=
a_{20} \cos^2\theta_1 + a_{02} \sin^2\theta_1\ne0,
$$
which implies that
the direction 
$\xi(\theta_1)$
is not an asymptotic direction of $f$ at the origin.
By the assumption $a_{02}\ne0$ and $0<\theta_1<\pi$,
it holds that
$$(\cos\theta_1h_u+\sin\theta_1h_v)_v(0,0)=a_{02}\sin\theta_1\ne0,$$
and this implies that there exists
a regular parametrization of $S$.
For the purpose of taking this parametrization,
we set an orientation of $S$ as follows.
First, we give an orientation of the normal plane $\xi(\theta_1)^\perp$ 
of $\xi(\theta_1)$ such that
$$
X=(-\sin\theta_1,\cos\theta_1,0),\quad
Y=(0,0,1)
$$
is a positive basis.
Next, we put an orientation on 
$(\pi_{\xi(\theta_1)}\circ f)(S)$
so that it agrees with the direction of $X$,
and also put that of $S$ 
agreeing with $(\pi_{\xi(\theta_1)}\circ f)(S)$ (see Figure \ref{fig:ori}).
\begin{figure}[!htb]
\centering
\includegraphics[scale=0.8, clip]{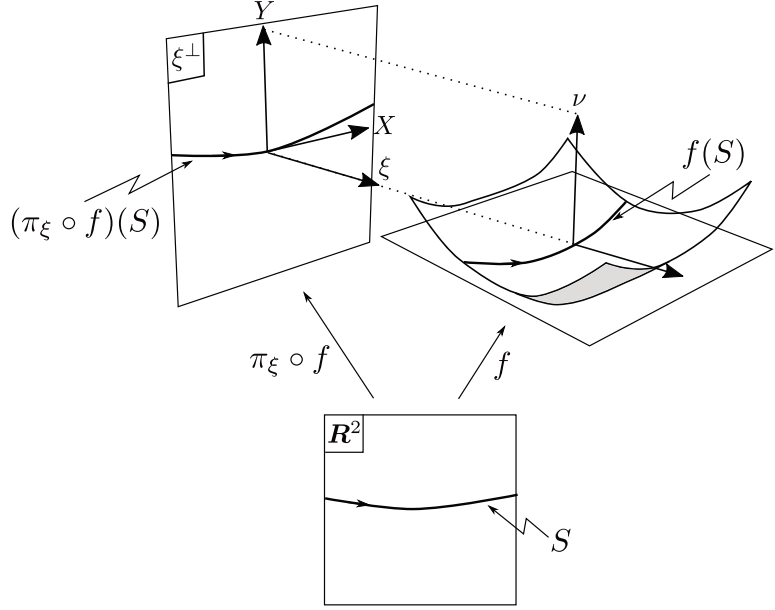}
\caption{Orientations of $\xi^{\perp}$ and the contour.\label{fig:ori}}
\end{figure}
Let $k_{\theta}$ be the curvature of 
the contour from a direction $\theta$.
\begin{lemma}
The first order 
information\/ $k_{\theta}$
with respect to the 
arc-length parameter\/ $s$ $($suitably oriented\/$)$ 
are
\begin{equation}\label{eq:surfcurv}
\coef(k_{\theta_1},1,s)
=
\Bigg(
\dfrac{a_{20}a_{02}}{p(\theta_1)},
\dfrac{q(\theta_1)}{p(\theta_1)^3}
\Bigg),
\end{equation}
where 
\begin{align}
q(\theta_1)
&=
a_{03} a_{20}^3 \cos^3\theta_1 
- 3 a_{02} a_{12} a_{20}^2 \cos^2\theta_1 \sin\theta_1 \nonumber\\
&\hspace{10mm}
+ 3 a_{02}^2 a_{20} a_{21} \cos\theta_1 \sin^2\theta_1 
- a_{02}^3 a_{30} \sin^3\theta_1.\label{eq:surfcurvq}
\end{align}
\end{lemma}
\begin{proof}
Since $a_{02}\sin\theta_1\ne0$,
we can take a parametrization 
$C(t)=(t,c(t))$ of $S$.
Then since $\pi_{\xi(\theta_1)}\circ f \circ C$ lies 
on the plane $\xi(\theta_1)^\perp$,
\begin{align*}
k_{\theta_1}(0)
=&\dfrac{\det\big(
\hat C',
\hat C'',\xi(\theta_1)\big)}
{|\hat C'|^3}(0),\\
\dfrac{dk_{\theta_1}}{ds}(0)
=&\dfrac{\det\big(
\hat C',
\hat C''',\xi(\theta_1)\big)}
{|\hat C'|^4}(0)
+
3\inner{\hat C'}{\hat C''}(0)
\dfrac{\det\big(
\hat C',
\hat C'',\xi(\theta_1)\big)}
{|\hat C'|^6}(0)
\end{align*}
where $\hat C=f\circ C=f(t,c(t))$ and ${}'=d/dt$.
Since $h'(t,c(t))=0$ at $t=0$, it holds that
$$
\det\big(
\hat C',
\hat C^{(i)},\xi(\theta_1)\big)=
(c'\cos\theta_1+\sin\theta_1)(h\circ C)^{(i)}\quad
(i=2,3),
$$
where $\hat C^{(2)}=\hat C''$ and $\hat C^{(3)}=\hat C'''$.
By \eqref{eq:projsingcond},
it holds that
\begin{align*}
\coef(c(t),2,t)
=&
\Bigg(
0,\;\;
-\dfrac{a_{20} \cos\theta_1}{a_{02}\sin\theta_1}
,\;\;
\dfrac{1}{a_{02}^3\sin^3\theta_1 }
\tilde{q}(\theta_1)\Bigg),
\end{align*}
where
\begin{align*}
\tilde{q}(\theta_1)=&
-a_{12} a_{20}^2 \cos^3\theta_1 
- a_{03} a_{20}^2 \cos^2\theta_1 \sin\theta_1
+ 2 a_{02} a_{20} a_{21} \cos^2\theta_1 \sin\theta_1 \\
+& 2 a_{02} a_{12} a_{20} \cos\theta_1 \sin^2\theta_1 
- a_{02}^2 a_{30} \cos\theta_1 \sin^2\theta_1 
- a_{02}^2 a_{21} \sin^3\theta_1.
\end{align*}
Summarizing up the above calculation, we have the assertion.
\end{proof}
The inner product of
$(\pi_{\xi(\theta_1)}\circ f\circ C)'(0)$ 
and $X$ is
\begin{equation}\label{eq:pixifc}
\inner{(\pi_{\xi(\theta_1)}\circ f\circ C)'(0)}{X}
=
\dfrac{-1}{a_{02}\sin\theta_1}p(\theta_1).
\end{equation}
Let $s$ be the arc-length parameter of $S$
where the orientation is given in the above manner.
Thus we remark that by \eqref{eq:pixifc}, if 
$a_{02}\sin\theta_1p(\theta_1)$ is negative,
$s$ is the opposite direction to the above parameter $t$.
\begin{remark}
If $a_{20} a_{02} \ne0$ and $p(\theta_1) \ne 0$, 
then $q(\theta_1)=0$ if and only if the contour has a vertex at 
$(\pi_{\xi(\theta_1)}\circ f\circ C)(0)$, 
and $\xi(\theta_1)=(\cos\theta_1,\sin\theta_1,0)$ 
is called the cylindrical direction of $f$ at 
the origin (see \cite{FHN2017} for details). 
\end{remark}

Now we consider obtaining the second order information
of the surface by the contours of the projections of
three distinct directions.
Let $f$ be the surface given by \eqref{eq:surfsiki}.
Since the mean and the Gaussian curvatures 
of $f$ are $(a_{20}+a_{02})/2$ and
$a_{20}a_{02}$ respectively,
if the mean and the Gaussian curvatures are obtained,
then one can conclude that the second order information
are obtained. Thus we consider obtaining the
mean and the Gaussian curvatures.
We set
$M$ and $G$ to be the mean and the Gaussian curvatures at $0$:
\begin{equation}\label{eq:defmg}
M=\frac{a_{20}+a_{02}}{2},\quad 
G=a_{20}a_{02}.
\end{equation}
We also take another direction $\theta_2(\ne\theta_1, 0<\theta_2<\pi)$ 
which satisfies
$p(\theta_2)\ne0$.
By \eqref{eq:surfcurv} we get
\begin{equation}\label{eq:cos2thetai}
\cos 2\theta_i
=
\frac{-2a_{20}a_{02}+(a_{20}+a_{02})k_{\theta_i}}
{(a_{02}-a_{20})k_{\theta_i}}\quad(i=1,2).
\end{equation}
Substituting these formulae into the trigonometric identity
$$
\cos^2 2(\theta_i-\theta_j)+\cos^2 2\theta_i+\cos^2 2\theta_j
-2\cos 2(\theta_i-\theta_j) \cos 2\theta_i \cos 2\theta_j
-1=0,
$$
$(i,j=1,2)$ we get an equation $P_{ij}(M,G)=0$ where
\begin{eqnarray}\label{angle}
P_{ij}(M,G)=&&
\left(M_{ij}^2-G_{ij}\cos^2 (\theta_i-\theta_j)\right)G^2
-2G_{ij}M_{ij}\sin^2(\theta_i-\theta_j)GM \nonumber \\
&&+G_{ij}^2\sin^4(\theta_i-\theta_j)M^2
+G_{ij}^2\cos^2 (\theta_i-\theta_j) \sin^2 (\theta_i-\theta_j)G \nonumber \\
=&&
(M,G)\;
Q_{ij}
\;
\pmt{G\\M}
+G_{ij}^2\cos^2 (\theta_i-\theta_j) \sin^2 (\theta_i-\theta_j)G
\end{eqnarray}
and
\begin{align}
M_{ij}=&\frac{k_{\theta_{i}}+k_{\theta_{j}}}{2},\ 
G_{ij}=k_{\theta_{i}}k_{\theta_{j}},\label{eq:defmgij}\\
Q_{ij}=&\pmt{ 
M_{ij}^2-G_{ij}\cos^2 (\theta_i-\theta_j) 
&-G_{ij}M_{ij}\sin^2(\theta_i-\theta_j)\\
-G_{ij}M_{ij}\sin^2(\theta_i-\theta_j)
&G_{ij}^2\sin^4(\theta_i-\theta_j)
}.\nonumber
\end{align}
Since $P_{ij}(M,G)=0$ is a quadratic curve,
generally the values of $G$ and $M$ should be determined by
the curvatures of the apparent contours from three distinct directions.

\begin{theorem}\label{thm:threegandm}
Take three distinct directions\/ 
$0<\theta_1, \theta_2, \theta_3<\pi$ 
satisfying\/ $p(\theta_i)\not=0$ and\/ 
$\sin 2\theta_1+\sin 2\theta_2+\sin 2\theta_3\ne0$.
Then\/ $G$ and\/ $M$ are given as follows:
\begin{equation}\label{eq:gandm}
G=
\dfrac{\det L}{\det V},\quad
M=
\dfrac{\det P}{\det V},
\end{equation}
where
\begin{align}
x_1=&\trans{
\left(
\dfrac{M_{12}}{G_{12}},
\dfrac{M_{23}}{G_{23}},
\dfrac{M_{31}}{G_{31}}
\right)},\nonumber\\
x_2=&\trans{
\left(
\dfrac{M_{12}^2-G_{12}\cos^2 (\theta_1-\theta_2)}
{G_{12}^2\sin^2 (\theta_1-\theta_2)},
\dfrac{M_{23}^2-G_{23}\cos^2 (\theta_2-\theta_3)}
{G_{23}^2\sin^2 (\theta_2-\theta_3)},\right.}\nonumber\\
&\hspace{50mm}
\left.\dfrac{M_{31}^2-G_{31}\cos^2 (\theta_3-\theta_1)}
{G_{31}^2\sin^2 (\theta_3-\theta_1)}
\right),\nonumber\\
x_3=&\trans{
\big(
\cos^2 (\theta_1-\theta_2),
\cos^2 (\theta_2-\theta_3),
\cos^2 (\theta_3-\theta_1)\big)},\label{eq:x3}\\
x_4=&\trans{
\big(
\sin^2 (\theta_1-\theta_2),
\sin^2 (\theta_2-\theta_3),
\sin^2 (\theta_3-\theta_1)\big)}\label{eq:x4}
\end{align}
and
$$
L=\big(x_1,x_3,x_4\big),\quad
P=\big(x_2,x_3,x_4\big),\quad
V=\big(x_1,x_2,x_4\big).
$$
\end{theorem}
Here $\trans{(~)}$ stands for the transpose of a matrix.
\begin{proof}
A triplet of the equations 
$P_{12}(M,G)=P_{23}(M,G)=P_{31}(M,G)=0$ 
is a system of equations
$$
W \pmt{G^2 \\ GM \\ M^2}=G b,
$$
where $W=(w_1,w_2,w_3)$
with
\begin{align}
w_1&=\pmt{
M_{12}^2-G_{12}\cos^2 (\theta_1-\theta_2) \\ 
M_{23}^2-G_{23}\cos^2 (\theta_2-\theta_3) \\ 
M_{31}^2-G_{31}\cos^2 (\theta_3-\theta_1)},
\label{eq:w1}\\
w_2&=-\pmt{
2G_{12}M_{12}\sin^2(\theta_1-\theta_2) \\ 
2G_{23}M_{23}\sin^2(\theta_2-\theta_3) \\ 
2G_{31}M_{31}\sin^2(\theta_3-\theta_1)},\ 
w_3=\pmt{
G_{12}^2\sin^4(\theta_1-\theta_2) \\ 
G_{23}^2\sin^4(\theta_2-\theta_3) \\ 
G_{31}^2\sin^4(\theta_3-\theta_1)},
\label{eq:w23}
\end{align}
and
\begin{equation}\label{eq:defb}
b=\pmt{
G_{12}^2\cos^2 (\theta_1-\theta_2) \sin^2 (\theta_1-\theta_2) \\ 
G_{23}^2\cos^2 (\theta_2-\theta_3) \sin^2 (\theta_2-\theta_3) \\ 
G_{31}^2\cos^2 (\theta_3-\theta_1) \sin^2 (\theta_3-\theta_1)}.
\end{equation}
Since 
$$
\det W
=
-2G_{12}^2G_{23}^2G_{31}^2
\sin^2 (\theta_1-\theta_2)
\sin^2 (\theta_2-\theta_3)
\sin^2 (\theta_3-\theta_1)
\det V,
$$
if $k_{\theta_1}k_{\theta_2}k_{\theta_3}\ne0$,
by Cramer's rule, 
we get
\begin{equation}\label{eq:gmw}
G
=
\dfrac{\det W_1}{\det W},\ 
M
=
\dfrac{\det W_2}{\det W},\ 
\end{equation}
where
$W_1=(b,w_2,w_3)$,
$W_2=(w_1,b,w_3)$.
By a direct calculation,
we have \eqref{eq:gandm}.
Furthermore, since by
\eqref{eq:cos2thetai}, it holds that
$k_i=2a_{20}a_{02}/\big(a_{20}+a_{02}+\cos2\theta_i(a_{20}-a_{02})\big)$,
we have
\begin{align*}
\det V=&
\dfrac{(a_{02}-a_{20})}{2a_{20}^2a_{02}^2}
\Big(
\sin(\theta_1-\theta_2)
\sin(\theta_2-\theta_3)
\sin(\theta_3-\theta_1)\\
&\hspace{40mm}
\big(\sin 2\theta_1+\sin 2\theta_2+\sin 2\theta_3\big)\Big).
\end{align*}
Thus we have the condition $\sin 2\theta_1+\sin 2\theta_2+\sin 2\theta_3
\ne0$.
\end{proof}

Since in our setting,
$\theta_1-\theta_2$, $\theta_1-\theta_3$,
$k_{\theta_1}$, $k_{\theta_2}$, $k_{\theta_3}$
are known, 
all matrix elements of $W$ and $b$ are known.
Thus we obtain $G$ and $M$ by \eqref{eq:gmw}.
Moreover, we obtain $\theta_1$, $\theta_2$ and $\theta_3$
by \eqref{eq:cos2thetai}.
Since $G=a_{20}a_{02}$ and $M=(a_{20}+a_{02})/2$, 
we obtain $a_{20}$ and $a_{02}$.
We remark that
the set
\begin{equation}\label{eq:sinsinsin}
\{
(\theta_1,\theta_2,\theta_3)\in(0,\pi)^3
\,|\,\sin 2\theta_1+\sin 2\theta_2+\sin 2\theta_3
\ne0\}
\end{equation}
is an open and dense subset of $(0,\pi)^3$.
On the other hand,
let us set $f_1(u,v)=(u,v,u^2+2v^2)$ and
$\theta_1=0,\theta_2=\pi/6,\theta_3=-\pi/6$.
Then $\sin 2\theta_1+\sin 2\theta_2+\sin 2\theta_3=0$.
Let $\xi(\theta)=(\cos\theta,\sin\theta,0)$ be a unit vector
and let $k^1_\theta$ be the curvature of the contour
from the direction $\theta$ about $f_1$.
Then $$
k^1_{\theta_1}=4,\quad
k^1_{\theta_2}=k^1_{\theta_3}=\dfrac{50}{7\sqrt{7}}.
$$
Let us set $f_2(u,v)=(u,v,au^2+2v^2)$ and
$\theta_1=0,\theta_2=\pi/4,\theta_3=-\pi/4$.
Then $\sin 2\theta_1+\sin 2\theta_2+\sin 2\theta_3=0$ also holds.
Let $k^2_\theta$ be the curvature of the contour
from the direction $\theta$ about $f_2$.
Then $$
k^2_{\theta_1}=4,\quad
k^2_{\theta_2}=k^2_{\theta_3}=\dfrac
{2 \sqrt{2} a (2 + a)^2}{(4 + a^2)^{3/2}},
$$
and there is a solution $a$ of the equation
\begin{equation}\label{eq:sola}
\dfrac
{2 \sqrt{2} a (2 + a)^2}{(4 + a^2)^{3/2}}-
\dfrac{50}{7\sqrt{7}}=0.
\end{equation}
In fact, the left hand side of \eqref{eq:sola}
is $2 (441 \sqrt{10}-625 \sqrt{7})/1225<0$ when $a=1$,
and it
is $2(98 - 25 \sqrt{7})/49>0$ when $a=2$.
This implies that the condition
$\sin 2\theta_1+\sin 2\theta_2+\sin 2\theta_3\ne0$
cannot be removed from the condition of 
Theorem \ref{thm:threegandm}.

Next let us consider the third order terms of the surface.
Let us take four distinct directions
$$
\theta_1,\ 
\theta_2,\ 
\theta_3,\ 
\theta_4.
$$
Then we obtain $a_{30},a_{21},a_{12},a_{03}$
by $k_{\theta_i}$ $(i=1,2,3,4)$ as follows.
By \eqref{eq:surfcurv} and \eqref{eq:surfcurvq},
we see that
$$
A
\pmt{a_{30}\\ a_{21} \\ a_{12} \\ a_{03}}
=
d,
$$
where $A=(a_1,a_2,a_3,a_4)$ and
\begin{align*}
a_1&=-a_{02}^3
\trans{\Big(
\sin^3\theta_1,\ 
\sin^3\theta_2,\ 
\sin^3\theta_3,\ 
\sin^3\theta_4\Big)},\\
a_2&=3a_{20}a_{02}^2\trans{
\Big(
\sin^2\theta_1\cos\theta_1,\ 
\sin^2\theta_2\cos\theta_2,\ 
\sin^2\theta_3\cos\theta_3,\ 
\sin^2\theta_4\cos\theta_4\Big)},\\
a_3&=-3a_{20}^2a_{02}\trans{
\Big(
\sin\theta_1\cos^2\theta_1,\ 
\sin\theta_2\cos^2\theta_2,\ 
\sin\theta_3\cos^2\theta_3,\ 
\sin\theta_4\cos^2\theta_4\Big)},\\
a_4&=a_{20}^3\trans{
\Big(
\cos^3\theta_1,\ 
\cos^3\theta_2,\ 
\cos^3\theta_3,\ 
\cos^3\theta_4\Big)},\\
d&=\trans{
\Big(
p(\theta_1)^3\dfrac{dk_{\theta_1}}{ds}(0),\ 
p(\theta_2)^3\dfrac{dk_{\theta_2}}{ds}(0),\ 
p(\theta_3)^3\dfrac{dk_{\theta_3}}{ds}(0),\ 
p(\theta_4)^3\dfrac{dk_{\theta_4}}{ds}(0)\Big)
}.
\end{align*}
Since
$\det A=9a_{20}^6a_{02}^6 \prod_{i<j}\sin(\theta_i - \theta_j)$,
and $\theta_1,\ldots,\theta_4$ are distinct, $a_{20}a_{02}\ne0$,
and it holds that
$\det A\ne0$.
By Cramer's rule, 
we get
\begin{equation}\label{eq:a3121}
a_{30}
=
\dfrac{\det A_1}{\det A},\ 
a_{21}
=
\dfrac{\det A_2}{\det A},\ 
a_{12}
=
\dfrac{\det A_3}{\det A},\ 
a_{03}
=
\dfrac{\det A_4}{\det A},
\end{equation}
where 
$A_1=(d,a_2,a_3,a_4)$,
$A_2=(a_1,d,a_3,a_4)$,
$A_3=(a_1,a_2,d,a_4)$,
$A_4=(a_1,a_2,a_3,d)$.
Thus we obtain $a_{30},a_{21},a_{12},a_{03}$
by $k_{\theta_i}$ $(i=1,2,3,4)$.

\subsection{Quadratic curves defined by two directions}
In this section, under the above setting,
we consider the quadratic curve
$$C=C_{12}(M,G)=\{(M,G)\in\R^2\,|\,P_{12}(M,G)=0\}$$
in the $MG$-plane, where
$P_{12}(M,G)$ is defined in \eqref{angle}.
This curve satisfies that 
for two points $(M,G)$ and $(\tilde{M},\tilde{G}) \in C$,
there exist surfaces $f=(u,v,(a_{20}u^2+a_{02}v^2)/2+O(3)$
and
$\tilde{f}=(u,v,(\tilde{a}_{20}u^2+\tilde{a}_{02}v^2)/2+O(3))$
($a_{20}+a_{02}=2M,a_{20}a_{02}=G,
\tilde{a}_{20}+\tilde{a}_{02}=2\tilde{M},
\tilde{a}_{20}\tilde{a}_{02}=\tilde{G}$),
and
there exist $\theta_1, \theta_2, \tilde{\theta}_1, \tilde{\theta}_2$
such that $\theta_1-\theta_2=\tilde{\theta}_1-\tilde{\theta}_2$
and $k_{\theta_i}=\tilde{k}_{\tilde{\theta}_i}$ ($i=1,2$),
where
$\tilde{k}_{\tilde{\theta}_i}$ is the curvature of 
the contour in the direction $\tilde{\theta}_i$ of the
surface $\tilde{f}$.
Since we assume that $a_{20}>a_{02}$ and $a_{20}>0$,
a point on $C$ expresses the unique surface up to second order information.
Hence we have a family of surfaces 
where curvatures of their contours with respect to
two (moving) directions
do not change.
\begin{proposition}
\label{prop:QuadraticCurve}
We have the following:
\begin{enumerate}
\item
The curve\/ $C$ is a hyperbola\/ $($respectively ellipse\/$)$ 
if\/ $G_{ij}>0$ $($respectively\/ $G_{ij}<0)$. 
\item
The curve\/ $C$ is tangent to the\/ $M$-axis. 
\item
If\/ $C$ is a hyperbola, then the two branches of\/ $C$ lie on opposite sides 
of\/ $\{G>0\}$ and\/ $\{G<0\}$.
\end{enumerate}
\end{proposition}
\begin{proof}
We have $\det Q_{ij}=-G_{ij}^3 \cos^2 (\theta_i-\theta_j) \sin^4 (\theta_i-\theta_j)$, 
and (1) is proved. 
We have $\partial P_{ij}/\partial M(0,0)=0$, which gives (2). 
By (1) and (2), it is clear that the assertion (3) holds.
\end{proof}
The assertion (3) of Proposition \ref{prop:QuadraticCurve} means that the above family
can be divided into
two continuous families
whose Gaussian curvatures are always positive and always negative
respectively.
\begin{example}\label{ex:contoursame}
Let us set
\begin{align*}
a_{20}&=3-\sqrt{3}+\sqrt{11-6\sqrt{3}},&
a_{02}&=3-\sqrt{3}-\sqrt{11-6\sqrt{3}},\\
\tilde{a}_{20}&=-6+\sqrt{37}, &
\tilde{a}_{02}&=-6-\sqrt{37},
\end{align*}
and let us set
\begin{align*}
\theta_1&=(-1/2)\arccos\left(\dfrac{1}{13}
\left(-4\sqrt{11-6\sqrt{3}}-\sqrt{3(11-6\sqrt{3})}\right)\right),\\
\tilde{\theta}_1&=-\dfrac{1}{2}\arccos\left(\dfrac{5}{\sqrt{37}}\right),\\
\end{align*}
$\theta_2=\theta_1+\pi/6$ and 
$\tilde{\theta}_2=\tilde\theta_1+\pi/6$.
Let $f$ and $\tilde{f}$ be two surfaces
defined by 
$$
f=\left(u,v,\dfrac{a_{20}u^2+a_{02}v^2}{2}\right),\quad
\tilde{f}=
\left(u,v,\dfrac{\tilde{a}_{20}u^2+\tilde{a}_{02}v^2}{2}\right).
$$
Then the curvatures of the contours of
$f$ with respect to $\xi(\theta_1)$ and
of
$\tilde{f}$ with respect to $\xi(\tilde{\theta}_1)$ are 1,
and
the curvatures of the contours of
$f$ with respect to $\xi(\theta_2)$ and
of
$\tilde{f}$ with respect to $\xi(\tilde{\theta}_2)$ are 2.
The $C$ of $f$, $\tilde{f}$ can be drawn as in 
Figure \ref{fig:fighyperb}.
See Figures \ref{fig:twocontour1}, \ref{fig:twocontour2}
and \ref{fig:twocontour3} for these surfaces
and their contours.

\begin{figure}[!ht]
\centering
\includegraphics[width=0.5\linewidth]{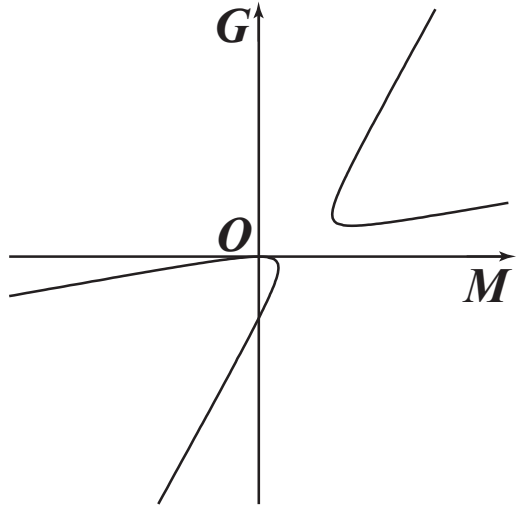}
\caption{The hyperbola $C$\label{fig:fighyperb}}
\end{figure}

\begin{figure}[!ht]
\centering
\includegraphics[width=0.25\linewidth]{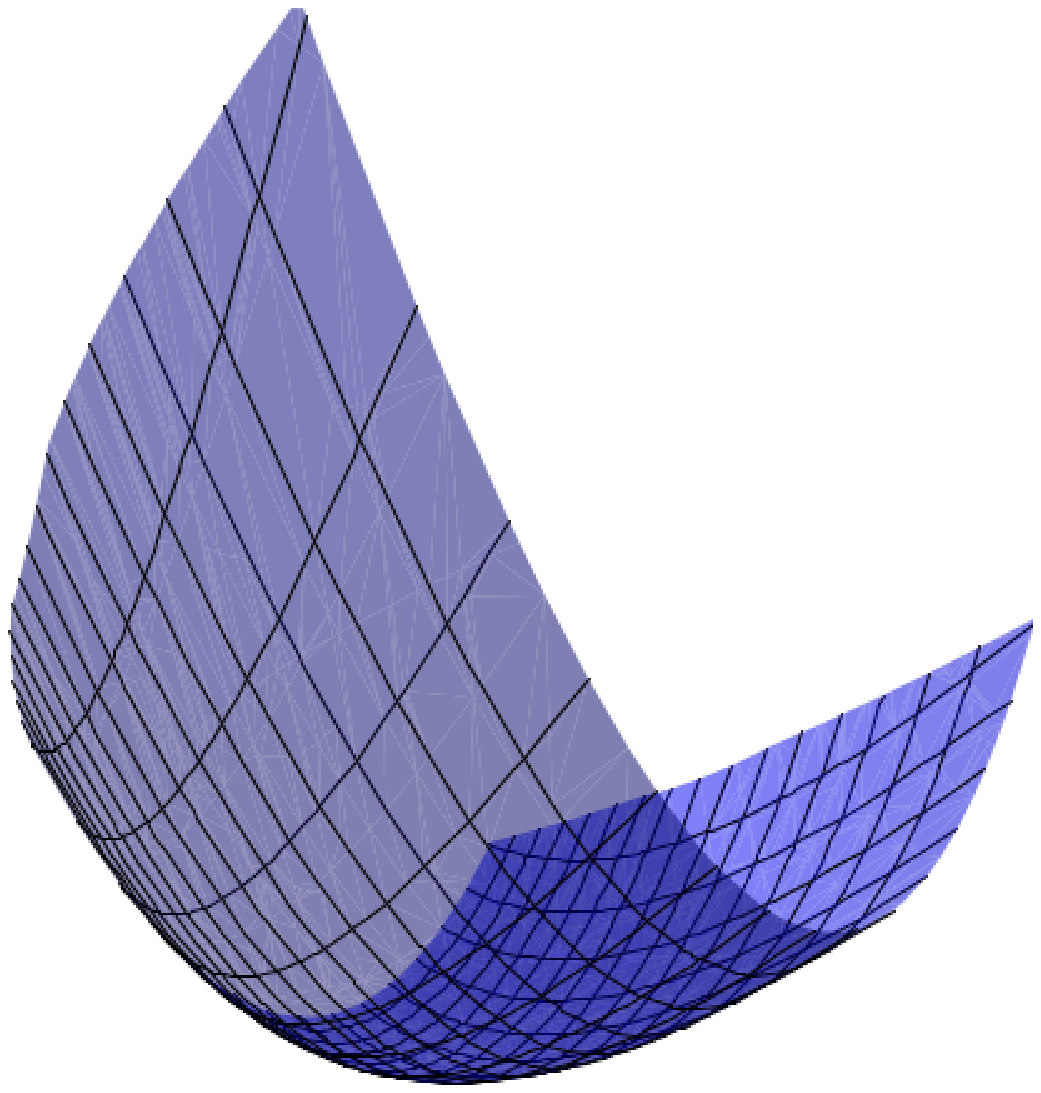}
\hspace{5mm}
\includegraphics[width=0.22\linewidth]{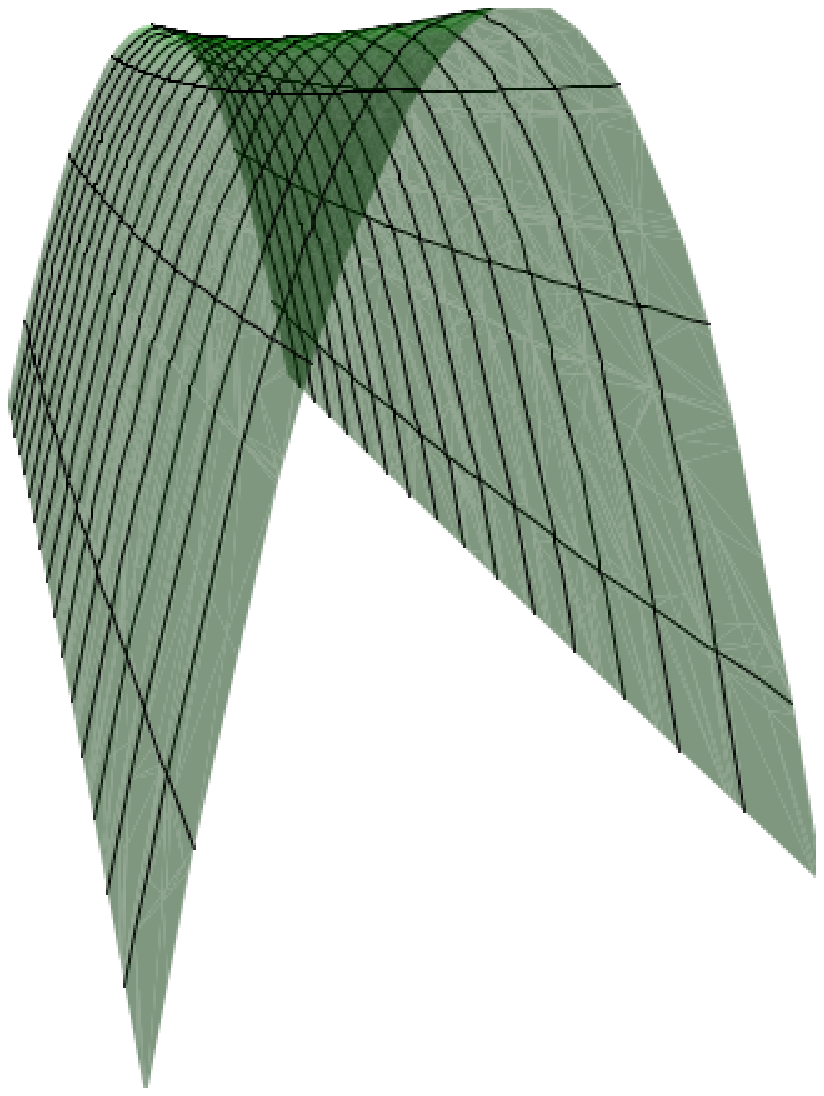}
\caption{The surfaces $f$ (blue) and $\tilde{f}$ (green) viewed from
$\xi(\theta_1)$ and $\xi(\tilde{\theta}_1)$, respectively. 
These contours have the same curvatures.\label{fig:twocontour1}}
\end{figure}

\begin{figure}[!ht]
\centering
\includegraphics[width=0.25\linewidth]{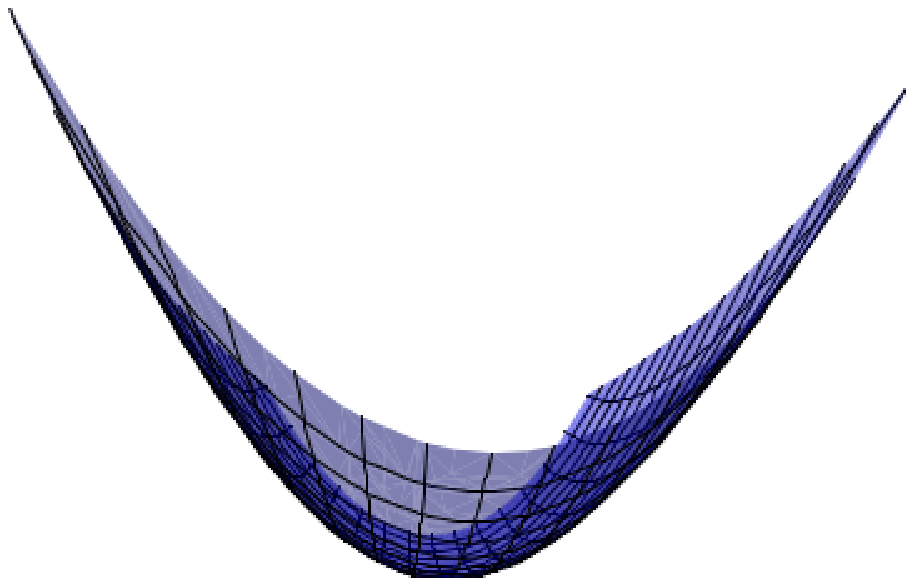}
\hspace{5mm}
\includegraphics[width=0.22\linewidth]{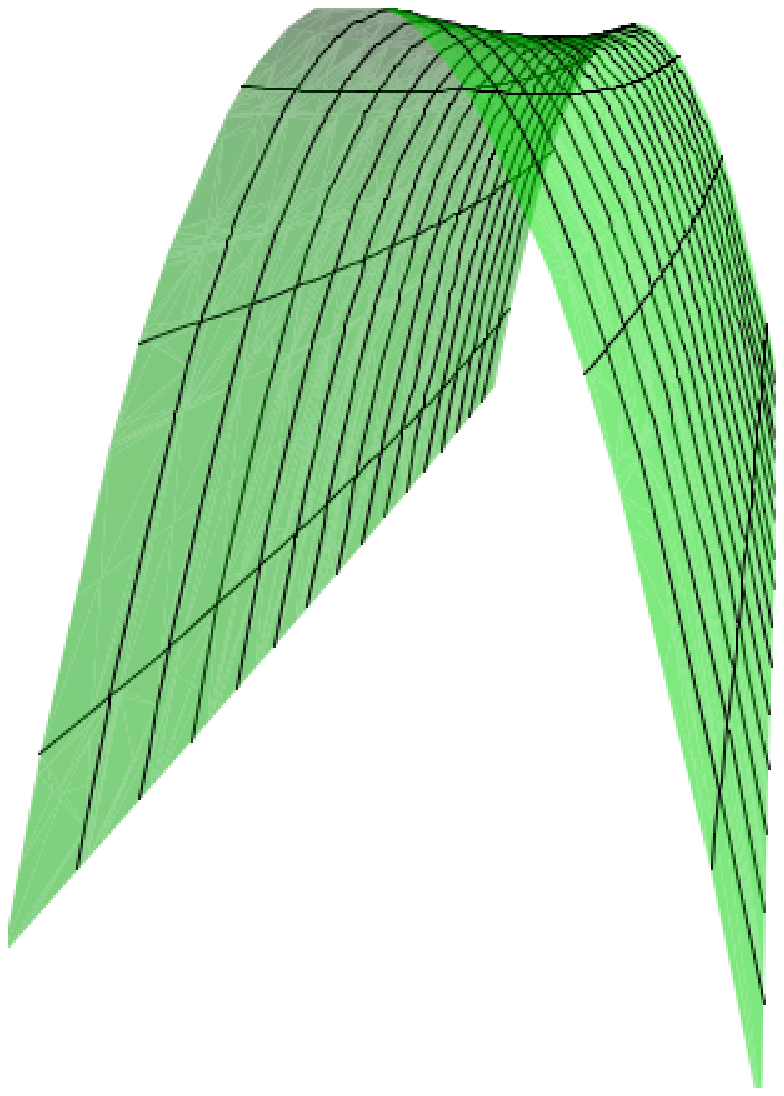}
\caption{The surfaces $f$ and $\tilde{f}$ viewed from
$\xi(\theta_2)$ and $\xi(\tilde{\theta}_2)$, respectively. 
These contours have the same curvatures.\label{fig:twocontour2}}
\end{figure}

\begin{figure}[!ht]\label{fig33}
\centering
\includegraphics[width=0.25\linewidth]{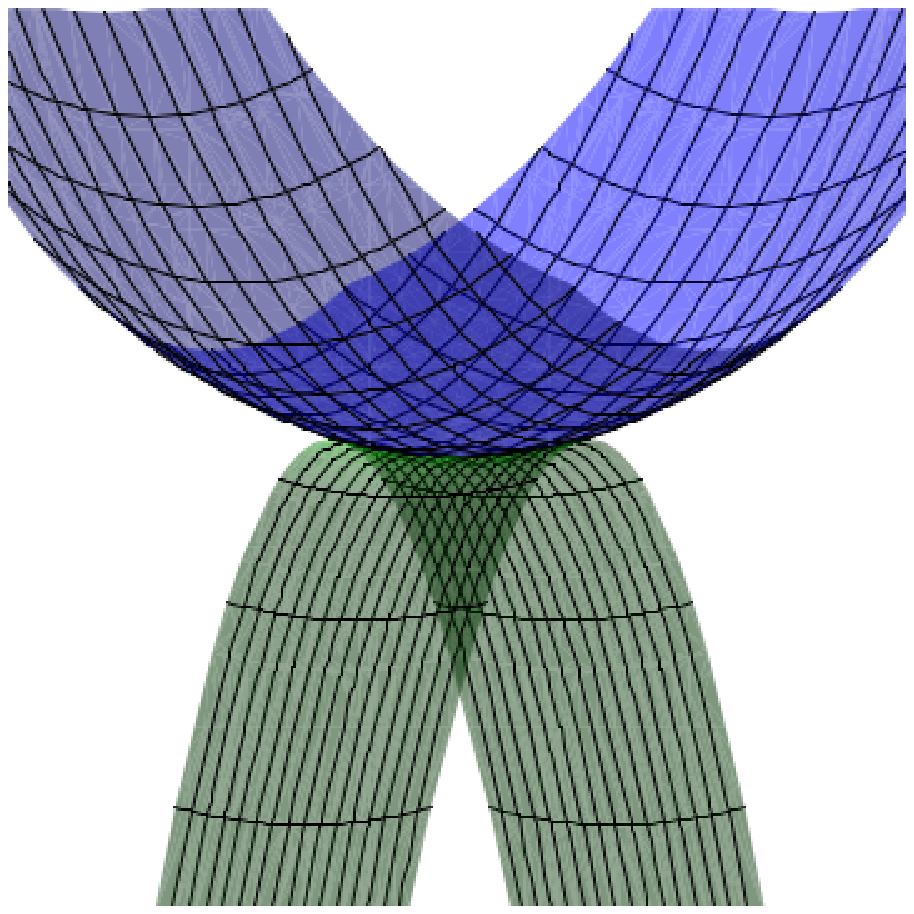}
\hspace{5mm}
\includegraphics[width=0.22\linewidth]{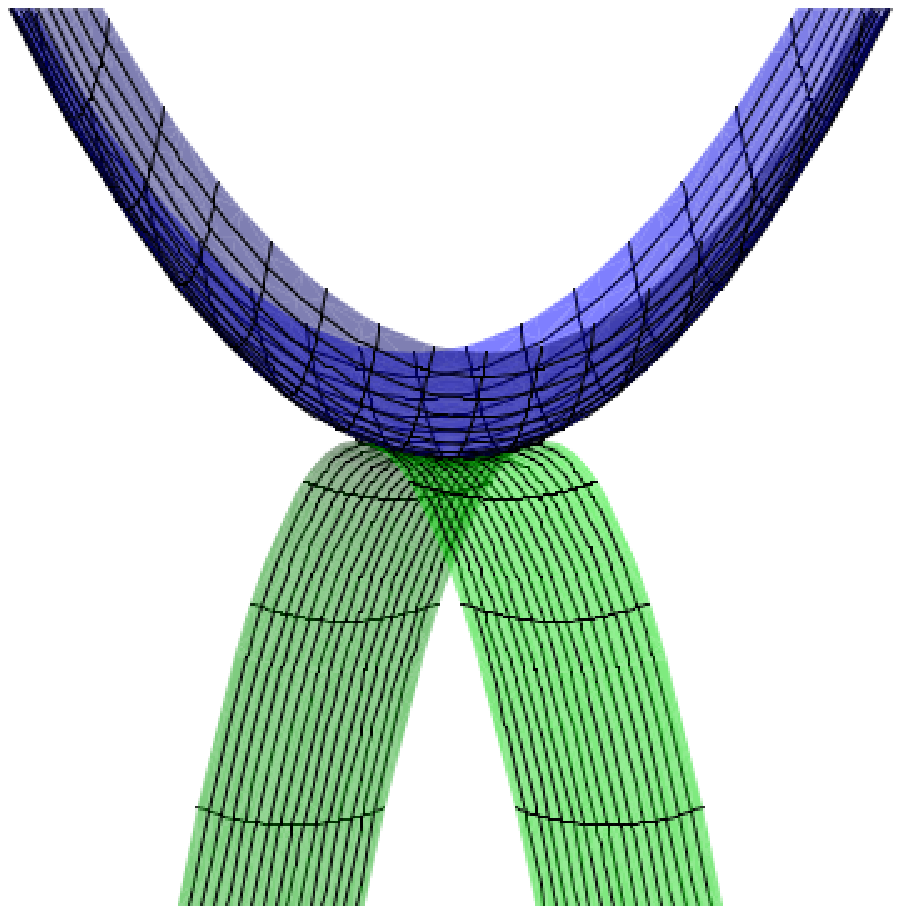}
\caption{
The surface $f$ rotated by angle $-\theta_1$ in the $XY$-plane with
the surface $\tilde{f}$ rotated by angle $-\tilde{\theta}_1$ in the $XY$-plane, and
the surface $f$ rotated by angle $-\theta_2$ in the $XY$-plane with
the surface $\tilde{f}$ rotated by angle $-\tilde{\theta}_2$ in the $XY$-plane, 
viewed from $\xi((1,0,0))$.\label{fig:twocontour3}}
\end{figure}
\end{example}

\subsection{Obtaining Gaussian curvature}
According to Section \ref{sec:surfnp},
we can obtain all of the second order 
information of the surface by 
the contour of projections from three distinct
directions.
In particular, we can obtain the Gaussian curvature.
In this section, we discuss
obtaining just the Gaussian curvature $K$.

By \eqref{eq:surfcurv},
we have 
$$
k_{\theta_1}k_{\theta_2}
=
\dfrac{a_{20}^2a_{02}^2}
{
(a_{20}\cos^2\theta_1+a_{02}\sin^2\theta_1)
(a_{20}\cos^2\theta_2+a_{02}\sin^2\theta_2)
}.
$$
Hence if 
\begin{equation}\label{eq:conj}
\dfrac{a_{20}a_{02}}
{
(a_{20}\cos^2\theta_1+a_{02}\sin^2\theta_1)
(a_{20}\cos^2\theta_2+a_{02}\sin^2\theta_2)
}=1,
\end{equation}
then $K=k_{\theta_1}k_{\theta_2}$.
If $\theta_1$, $\theta_2$ satisfy \eqref{eq:conj},
then we say that
$\xi_{\theta_1},\xi_{\theta_2}$ are {\it contour-conjugate\/}
each other.
Now we consider the existence of the contour-conjugate.
Since \eqref{eq:conj} is equivalent to
\begin{equation}\label{eq:conj2}
a_{20}-a_{02}=0\quad\text{or}\quad
a_{20}\dfrac{\cos^2\theta_2}{\sin^2\theta_2}
=
a_{02}\dfrac{\sin^2\theta_1}{\cos^2\theta_1},
\end{equation}
we have the following proposition. 
\begin{proposition}\label{prop:contconj}
Let\/ $p$ be a point that is not flat umbilic on
a regular surface.
If\/ $p$ is an umbilic point, then any pair of two directions
are contour-conjugates at\/ $p$.
If\/ $K(p)>0$ and\/ $p$ is not an umbilic point, 
then any direction has two contour-conjugates
at\/ $p$,
and
if\/ $K(p)<0$ there are no contour-conjugate for any direction
at\/ $p$.
\end{proposition}
\begin{example}
Let us set
$$
f(u,v)=\left(u,v,\dfrac{u^2}{2}+v^2\right)
$$
and
$$
\theta_1=\pi/4,\quad
\theta_2=\operatorname{arccot}(\sqrt{2}).
$$
Then, 
since $\theta_1$ and $\theta_2$ satisfy
\eqref{eq:conj2},
they are contour-conjugate (see Figure \ref{fig:contourconj}),
namely, 
the product 
of the curvatures with respect to these directions
equals $2$, the Gaussian curvature of $f$ at $0$.

\begin{figure}[!ht]
\centering
\includegraphics[width=0.3\linewidth]{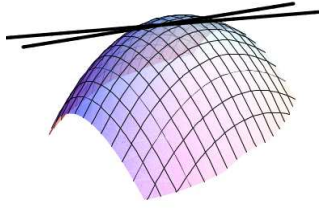}
\caption{Contour-conjugate directions\label{fig:contourconj}}
\end{figure}
\end{example}

\section{Normal curvature and Euler's formula}
In this appendix, we give 
a similar formula to Theorem \ref{thm:threegandm}
for the normal curvatures.
To obtain $a_{20},a_{02}$,
we have another expression by using
Euler's formula (see \cite[p 214]{oneill}, for example).
In the same setting as in Section \ref{sec:surfnp},
let $\theta_1$, $\theta_2$, $\theta_3$ be
the distinct angles $(0<\theta_1,\theta_2,\theta_3<\pi)$. 
Let $k_i^n$ $(i=1,2,3)$ be
the normal curvatures of $f$ with respect to $\xi(\theta_i)$,
and let $M^n_{ij}=(k_i^n+k_j^n)/2$, $G^n_{ij}=k_i^nk_j^n$.
By Euler's formula, 
we have
\begin{align}\label{EL}
\sin^4(\theta_i-\theta_j)M^2
-2M_{ij}^n\sin^2(\theta_i-\theta_j) M
+\cos^2 (\theta_i-\theta_j) \sin^2 (\theta_i-\theta_j)G\\
\hspace{15mm}
+\left((M_{ij}^n)^2-G_{ij}^n\cos^2 (\theta_i-\theta_j)\right)
=0.\nonumber
\end{align}
For $ij=12,23,31$, these equations form
a linear system
\begin{align*}
&\pmt{
\sin^4(\theta_1-\theta_2) &
-2M_{12}^n\sin^2(\theta_1-\theta_2) &
\cos^2 (\theta_1-\theta_2) \sin^2 (\theta_1-\theta_2) \\ 
\sin^4(\theta_2-\theta_3) &
-2M_{23}^n\sin^2(\theta_2-\theta_3) &
\cos^2 (\theta_2-\theta_3) \sin^2 (\theta_2-\theta_3) \\ 
\sin^4(\theta_3-\theta_1)&
-2M_{31}^n\sin^2(\theta_3-\theta_1)&
\cos^2 (\theta_3-\theta_1) \sin^2 (\theta_3-\theta_1)}
\pmt{M^2 \\ M \\ G}\\
&\hspace{10mm}
=\pmt{
(M_{12}^n)^2-G_{12}^n\cos^2 (\theta_1-\theta_2) \\ 
(M_{23}^n)^2-G_{23}^n\cos^2 (\theta_2-\theta_3) \\ 
(M_{31}^n)^2-G_{31}^n\cos^2 (\theta_3-\theta_1)}.
\end{align*}
By Cramer's rule, 
we have the expressions 
$$
M=
\dfrac{\det P^n}{\det V^n},\quad
G=
\dfrac{2\det L^n}{\det V^n}
$$
under the condition $\det V^n\ne0$,
where
$
V^n=\big(x_5,x_3,x_4\big)$,
$L^n=\big(x_6,x_5,x_4\big)$,
$P^n=\big(x_6,x_3,x_4\big)$,
$x_3,x_4$ are in \eqref{eq:x3}, \eqref{eq:x4}
respectively,
and
\begin{align*}
x_5&=\trans{\big(
M_{12}^n,M_{23}^n,M_{31}^n\big)}\\
x_6&=\trans{\left(
\dfrac{(M_{12}^n)^2-G_{12}^n\cos^2 (\theta_1-\theta_2)}
{\sin^2 (\theta_1-\theta_2)},
\dfrac{(M_{23}^n)^2-G_{23}^n\cos^2 (\theta_2-\theta_3)}
{\sin^2 (\theta_2-\theta_3)},\right.}\\
&\hspace{60mm}
\left.\dfrac{(M_{31}^n)^2-G_{31}^n\cos^2 (\theta_3-\theta_1)}
{\sin^2 (\theta_3-\theta_1)}\right).
\end{align*}
\section*{Competing Interests}
The authors declare that they have no competing
interests.
\section*{Author's Contributions}
The current paper was jointly developed by the three authors 
in the seminar discussions. All the authors equally contributed.


\medskip
{\footnotesize
\begin{flushright}
\begin{tabular}{l}
(Hasegawa)\\
Depart of Information Science,\\
Center for Liberal Arts and Sciences,\\
Iwate Medical University, \\
2-1-1 Idaidori, Yahaba-cho, Shiwa-gun, Iwate,\\
028-3694, Japan\\
E-mail: {\tt mhaseO\!\!\!aiwate-med.ac.jp}\\
\\
(Kabata)\\
School of Information and Data Sciences,\\
Nagasaki University, \\
Bunkyocho 1-14, Nagasaki\\
852-8131, Japan\\
E-mail: {\tt kabataO\!\!\!anagasaki-u.ac.jp}\\
\\
(Saji)\\
Department of Mathematics,\\
Graduate School of Science, \\
Kobe University, \\
Rokkodai 1-1, Nada, Kobe \\
657-8501, Japan\\
  E-mail: {\tt sajiO\!\!\!amath.kobe-u.ac.jp}
\end{tabular}
\end{flushright}}
\end{document}